\documentclass{amsart}

\usepackage{hyperref}
\usepackage{srcltx}
\usepackage{graphicx}
\usepackage{amssymb}
\usepackage{amsmath}

\newcommand{\Set}[2]{\left\{\, #1 \;\middle|\; #2 \,\right\}}%set

\newtheorem{theorem}{Theorem}[section]
\newtheorem{lemma}[theorem]{Lemma}
\theoremstyle{definition}

\newtheorem{proposition}{Proposition}
\newtheorem{corollary}[theorem]{Corollary}

\DeclareMathOperator{\size}{{size}}
\def\P{{\mathbf{P}}}

\def\NP{{\mathbf{NP}}}

\def\SSP{{\mathbf{SSP}}}

\def\SMP{{\mathbf{SMP}}}
\def\BSMP{{\mathbf{BSMP}}}
\def\BKP{{\mathbf{BKP}}}
\def\KP{{\mathbf{KP}}}

\def\AGP{{\mathbf{AGWP}}}
\def\dumb{{simple}}
\def\nondumb{{non-simple}}
\def\pbk{{polynomially bounded knapsack}}

\begin{document}

\title{Knapsack problems in products of groups}

\author[]{Elizaveta Frenkel, Andrey Nikolaev, and Alexander Ushakov}

\address{Elizaveta Frenkel, Moscow State University, GSP-1, Leninskie gory, 119991,
Moscow, Russia}
\email{lizzy.frenkel@gmail.com}

\address{Andrey Nikolaev, Stevens Institute of Technology, Hoboken, NJ, 07030 USA}
\email{anikolae@stevens.edu}

\address{Alexander Ushakov, Stevens Institute of Technology, Hoboken, NJ, 07030 USA}
\email{aushakov@stevens.edu}

\thanks{The work of the third author was partially supported by NSF grant DMS-1318716.}

\begin{abstract}
The classic knapsack and related problems have natural generalizations to arbitrary (non-commutative) groups, collectively called knapsack-type problems in groups. We study the effect of free and direct products on their time complexity. We show that free products in certain sense preserve time complexity of knapsack-type problems, while direct products may amplify it. Our methods allow to obtain complexity results for rational subset membership problem in amalgamated free products over finite subgroups.


\end{abstract}

\maketitle

\section{Introduction}\label{sec:intro}

%\subsection{Motivation}\label{sub:motivation}
In~\cite{Miasnikov-Nikolaev-Ushakov:2014a}, the authors introduce a number of certain decision, search and optimization algorithmic problems in groups, such as the subset sum problem, the knapsack problem, and the bounded submonoid membership problem (see Section~\ref{sec:prelim} for definitions). These problems are collectively referred to as {\em knapsack-type} problems and deal with different generalizations of the classic knapsack and subset sum problems over~$\mathbb Z$ to the case of arbitrary groups. In the same work, the authors study time complexity of such problems for various classes of groups, for example for nilpotent, metabelian, hyperbolic groups. With that collection of results in mind, it is natural to ask what is the effect of group constructions on the complexity of knapsack-type problems, primarily the subset sum problem. In the present paper we address this question in its basic variation, for the case of free and direct products of groups.

%In this paper, following~\cite{Miasnikov-Nikolaev-Ushakov:2014a}  we continue the research on  non-commutative discrete (combinatorial) optimization. Specifically, we set out to study the properties of the subset sum problem in free products of groups.
Solutions to many algorithmic problems carry over from groups to their free products without much difficulty. It certainly is the case with classic decision problems in groups such as the word, conjugacy~\cite[for instance]{Lyndon-Schupp:2001} and membership~\cite{Mikhailova_68} problems. In some sense, the same expectations are satisfied with knapsack-type problems, albeit not in an entirely straightforward fashion. It turns out that knapsack-type problems such as the aforementioned subset sum problem, the bounded knapsack problem, and the bounded submonoid membership problem share a certain common ground that allows to approach these problems in a unified fashion, and to carry solutions of these problems over to free products. Thus, our research both presents certain known facts about these algorithmic problems in a new light, and widens the class of groups with known complexity of the knapsack-type problems. Our methods apply more generally, which allows us to establish in Section~\ref{sec:free_prod} complexity results for certain decision problems, including the rational subset membership problem, in free products of groups with finite amalgamated subgroups.

Algorithmic problems in a direct product of groups can be dramatically more complex than in either factor, as is the case with the membership problem, first shown in~\cite{Mikhailova}. By contrast, the word and conjugacy problems in direct products easily reduce to those in the factors. %{\li{Reference? \cite{Lyndon-Schupp:2001}?}}  {\an{I don't think we need one here.}}{\li{ok.}}
In Section~\ref{sec:direct_prod} we show that direct product does not preserve polynomial time subset sum problem (unless $\P=\NP$). Thus, the subset sum problem occupies an interesting position, exhibiting features of both word problem and membership problem; on the one hand, its decidability clearly carries immediately from factors to the direct product, while, on the other hand, its time complexity can increase dramatically. % The now-classic construction in the latter work was recently exploited in~\cite{Miasnikov-Nikolaev-Ushakov:2014a} to show that a direct product does not preserve polynomial time bounded submonoid membership problem, unless $\P=\NP$.

Below we provide basic definitions and some of the immediate properties of the problems mentioned above.
%We refer to~\cite{Miasnikov-Nikolaev-Ushakov:2014a,Miasnikov-Nikolaev-Ushakov:2014b} for the initial motivation for the study of non-commutative discrete optimization, the set-up of the problems, and initial facts on non-commutative discrete optimization.

\subsection{Preliminaries}\label{sec:prelim}\label{sec:problems}

In this paper we follow terminology and notation introduced in \cite{Miasnikov-Nikolaev-Ushakov:2014a}.
For convenience, below we formulate the algorithmic problems mentioned in Section~\ref{sec:intro}. We collectively refer to these problems as {\it knapsack-type} problems in groups.

Elements in a group $G$ generated by a finite or countable set $X$ are given as words over the alphabet  $X \cup X^{-1}$. As we explain in the end of this section, the choice of a finite $X$ does not affect complexity of the problems we formulate below. Therefore, we omit the generating set from notation.
%We explain in the end of this section that the choice of a finite $X$ is, in fact, immaterial, but for the sake of rigor we keep dependence on $X$ in the below definitions.

Consider the following decision problem. Given $g_1,\ldots,g_k,g\in G$, and $m$ that is either a unary positive integer or the symbol $\infty$, decide if
  \begin{equation} \label{eq:general-def}
  g = g_1^{\varepsilon_1} \ldots g_k^{\varepsilon_k}
  \end{equation}
for some integers $\varepsilon_1,\ldots,\varepsilon_k$ such that $0\le \varepsilon_j \le m$ for all $j=1,2,\ldots,k$. Depending on $m$, special cases of this problem are called:
\begin{description}
\item[($m=\infty$)] {\bf The knapsack problem $\KP(G)$\index{$\KP(G)$}}. We omit $\infty$ from notation, so the input of $\KP(G)$ is a tuple $g_1,\ldots,g_k,g\in G$.
\item[($m<\infty$)] {\bf The bounded knapsack problem $\BKP(G)$\index{$\BKP(G)$}}. The input of  $\BKP(G)$ is a tuple $g_1,\ldots,g_k,g\in G$, and a number $1^m$.
\item[($m=1$)] {\bf The subset sum  problem $\SSP(G)$\index{$\SSP(G)$}}. We omit $m=1$ from notation, so the input of  $\SSP(G)$ is a tuple $g_1,\ldots,g_k,g\in G$.
\end{description}
One may note that $\BKP(G)$ is $\P$-time equivalent to  $\SSP(G)$ (we recall the definition of $\P$-time reduction below), so it suffices for our purposes to consider only $\SSP$ in groups.

\medskip
The submonoid membership problem formulated below is equivalent to the knapsack problem in the classic (abelian) case, but in general it is a completely different problem that is of prime  interest in algebra.

Consider the following decision problem. Given $g_1,\ldots,g_k,g\in G$, and $m$ that is either a unary positive integer or the symbol $\infty$, decide if the following equality holds for some $g_{i_1}, \ldots, g_{i_s} \in \{g_1, \ldots, g_k\}$ and some integer $s$ such that $0\le s\le m$:
\begin{equation}\label{eq:SMP-def}
g = g_{i_1} \cdots g_{i_s}.
\end{equation}
Depending on $m$, special cases of this problem are called:
\begin{description}
\item[($m=\infty$)] {\bf Submonoid membership problem $\mathbf{SMP}(G)$\index{$\SMP(G)$}}. We omit $\infty$ from notation, so the input of $\SMP(G)$ is a tuple $g_1,\ldots,g_k,g\in G$.
\item[($m<\infty$)] {\bf Bounded submonoid membership problem $\BSMP(G)$\index{$\BSMP(G)$}}. The input of this problem is a tuple $g_1,\ldots,g_k,g\in G$ and a number $1^m$.
\end{description}
%{\bf Submonoid membership problem $\mathbf{SMP}(G,X)$\index{$\SMP(G,X)$}}:  Given elements $g_1,\ldots,g_k,g\in G$
%decide if $g$ belongs to the submonoid generated by $g_1, \ldots, g_k$ in $G$, i.e., if the following equality holds for some $g_{i_1}, \ldots, g_{i_s} \in \{g_1, \ldots, g_k\}, s \in \mathbb{N}$:
%\begin{equation}\label{eq:SMP-def}
%g = g_{i_1}, \ldots, g_{i_s}.
%\end{equation}

\medskip
The restriction of $\SMP$ to the case when the set of generators $\{g_1, \ldots,g_n\}$ is closed under inversion (so the submonoid is actually a subgroup of $G$) is  a well-known problem in group theory, the {\em uniform subgroup membership problem} in $G$ (see e.g. \cite{MSU_book:2011}).
%There is a huge bibliography on this subject, we mention some related results  in Section \ref{sec:results}.

In general, both $\SMP$ and $\KP$ can be undecidable in a group with decidable word problem, for instance, a group with decidable word problem, but undecidable membership in cyclic subgroups is constructed in \cite{Olsh-Sapir}. Groups with undecidable $\SMP$ or $\KP$ are not necessarily quite so hand-crafted; for example, famous Mikhailova construction~\cite{Mikhailova_68} shows that subgroup membership and, therefore, submonoid membership is undecidable in a direct product of a free group with itself, and recent works show that the knapsack problem is undecidable in integer unitriangular groups of sufficiently large size~\cite{Lohrey:2015} and, more broadly, large family of nilpotent groups of class at least $2$~\cite{Misch-Treyer}. {\em Bounded} versions of $\KP$ and $\SMP$ are at least always decidable in groups where the word problem is. %In this case the problem is to verify  if the corresponding equalities (\ref{eq:IKP-def}) and  (\ref{eq:SMP-def}) hold for a given $g$ provided that the number of factors in these equalities  is bounded by a natural number $m$ which is given in the unary form, i.e., as the word $1^m$. In particular, we define the bounded knapsack problem ($\BKP$) for a group $G$ as follows.

Recall that a decision problem can be denoted as $(I,D)$, where $I$ is the set of all instances of the problem, and $D\subseteq I$ is the set of positive ones. A decision problem $(I_1,D_1)$ is  {\em $\P$-time  reducible}
to  a problem $(I_2,D_2)$ if there is a $\P$-time computable function
$f:I_1  \to I_2$  such that for any $u \in I_1$ one has $u \in D_1 \Longleftrightarrow f(u) \in D_2$.
Such reductions are usually called either {\em many-to-one} $\P$-time reductions or {\em Karp} reductions.
Since we mainly use this type of reduction we omit ``many-to-one'' from the name and call them $\P$-time reductions. We say that two problems are {\em $\P$-time equivalent} if each of them $\P$-time reduces to the other. Aside from many-to-one reductions, we use the so-called Cook reductions. That is, we say that a decision problem $(I_1,D_1)$ is  {\em $\P$-time Cook  reducible} to  a problem $(I_2,D_2)$ if there is an algorithm that solves problem $(I_1,D_1)$ using a polynomial number of calls to a subroutine for problem $(I_2,D_2)$, and polynomial time outside of those subroutine calls. Correspondingly, we say that two problems are {\em $\P$-time Cook equivalent} if each of them $\P$-time Cook reduces to the other. In our present work we are primarily interested in establishing whether certain problems are $\P$-time decidable or $\NP$-complete, and thus we make no special effort to use one of the two types of reduction over the other.

Finally we mention that for $\Pi\in \{\SSP, \KP, \BKP, \SMP, \BSMP\}$, the problem $\Pi(G,X_1)$ is $\P$-time (in fact, linear time) equivalent to $\Pi(G,X_2)$ for any finite generating sets $X_1,X_2$, i.e. the time complexity of $\Pi(G)$ does not depend on the choice of a finite generating set of a given group $G$. For this reason, we simply write $\Pi(G)$ instead of $\Pi(G,X)$, implying an arbitrary finite generating set.
We refer the reader to~\cite{Miasnikov-Nikolaev-Ushakov:2014a} for the proof of the above statement, further information concerning knapsack-type problems and their variations in groups, details regarding their algorithmic set-up, and the corresponding basic facts. Here we limit ourselves to mentioning that all of the above problems can be regarded as special cases of the uniform rational subset membership problem for $G$, with input given by a (non-deterministic) finite state automation and a group element. We take advantage of this viewpoint in Sections~\ref{sec:agp} and~\ref{sec:free_prod}.

\subsection{Results and open questions}\label{sub:results}
Primary goal of the present work is to answer the following basic questions about complexity of the knapsack problem ($\KP$) and the subset sum problem ($\SSP$) in free and direct products.
\begin{itemize}
\item[Q1.] Assuming $\SSP(G),\SSP(H)\in\P$, is it true that $\SSP(G*H)\in\P$?
\item[Q2.] Assuming $\KP(G),\KP(H)\in\P$, is it true that $\KP(G*H)\in\P$?%If the answer to the above question is positive, is it true that $\KP(G*H)\in\P$?
\item[Q3.] Assuming $\SSP(G),\SSP(H)\in\P$, is it true that $\SSP(G\times H)\in\P$?
\item[Q4.] Assuming $\KP(G),\KP(H)\in\P$, is it true that $\KP(G\times H)\in\P$?%If the answer to the above question is positive, is it true that $\KP(G\times H)\in\P$?
\end{itemize}

We give a partial positive answer to the question Q1. To elaborate, in the context of studying properties of the subset sum problem in free products, it is natural to view this problem, along with the bounded knapsack and the bounded submonoid membership problems,
as special cases of a more general algorithmic problem formulated in terms of graphs labeled by group elements, which we call the {\em acyclic graph word problem} ($\AGP$), introduced in Section~\ref{sec:agp} below.

In the same section we also establish connection between $\SSP$, $\BKP$, $\BSMP$ and $\AGP$, as shown in Figure~\ref{fi:reductions},
\begin{figure}[h]
 \centering
 \includegraphics{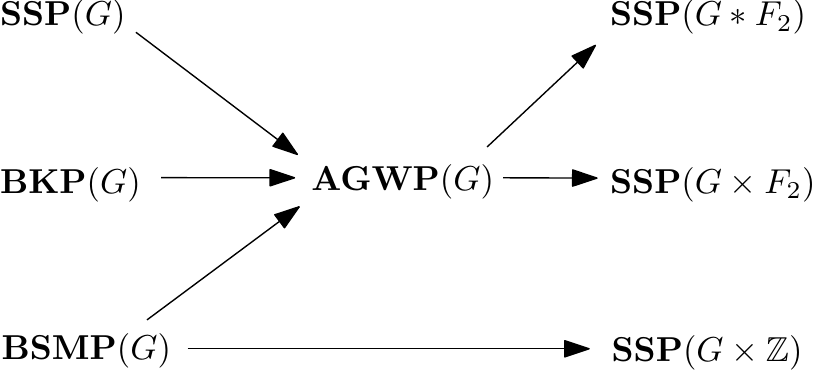}
 \caption{Connection between algorithmic problems. Arrows denote $\P$-time (Karp or Cook) reductions.}\label{fi:reductions}
\end{figure}
and show that $\AGP$ is $\P$-time solvable in all known groups where $\SSP$ is. Further, we show in %Theorem~\ref{th:agp_to_factors}
Section~\ref{sec:free_prod} that $\AGP$ in free products with amalgamation over a finite subgroup is $\P$-time Cook reducible to $\AGP$ in the factors. % As an immediate consequence, we assert in Corollary~\ref{cor:ptime_agp} that free products preserve polynomial time $\AGP$, i.e. $\AGP(G* H)\in\P$ if and only if $\AGP(G)\in\P$ and $\AGP(H)\in\P$.
As a consequence, we obtain that $\SSP$, $\BKP$, $\BSMP$ are polynomial time solvable in a wide class of groups. Tools used in Section~\ref{sec:free_prod} allow to produce similar complexity estimates for a wider class of problems, including rational subset membership problem.

As an answer to Q2, we show in Section~\ref{sec:knapsack} that the (unbounded) knapsack problem in a free product has a polynomial bound on length of solution if and only if $\KP(G)$ and $\KP(H)$ have such a bound. This condition allows a $\P$-time reduction of $\KP(G*H)$ to $\AGP(G*H)$, and, therefore, a $\P$-time solution of $\KP(G*H)$ when $\AGP(G)$, $\AGP(H)\in \P$.

%$\P$-time reduces to $\BKP$, which allows to conclude that, assuming $\AGP(G)$, $\AGP(H)\in\P$, the knapsack problem $\KP(G*H)$ has a polynomial bound on length of solution if and only if  $\P$-time solvable if and only if $\KP(G)$ and $\KP(H)$ have such a bound. %(Theorem~\ref{th:ptime_kp}).

We observe that the same is false for direct products, i.e. that $\AGP(G)\in\P$, $\AGP(H)\in\P$ does not imply $\AGP(G\times H)\in\P$ (unless $\P=\NP$). In Section~\ref{sec:direct_prod}
%Proposition~\ref{pr:ssp_cross}
we prove a similar result for $\SSP(G\times H)$, which negatively resolves question Q3.

%In Propositions~\ref{pr:agp_to_ssp_star} and~\ref{pr:agp_to_ssp_cross}, we establish reduction of $\AGP$ to $\SSP$, albeit in a different group (augmented by a free or direct factor). As an application of this idea, in Proposition~\ref{pr:ssp_cross} we answer one of the starting questions of this paper by showing that direct product does not preserve polynomial $\SSP$ (unless $\P=\NP$).
%, i.e. that there are groups $G,H$ such that $\SSP(G),\SSP(H)\in \P$, but $\SSP(G\times H)$ is $\NP$-complete (Corollary~\ref{cor:NP_complete_ssp_cross}).

The following questions remain open.
\begin{itemize}
\item[OQ1.] Is there a group $G$ with $\SSP(G)\in\P$ and $\NP$-hard $\AGP(G)$ (cf. results shown in Figure~\ref{fi:reductions})?
\item[OQ2.] Let $F_2$ be a free group of rank $2$. Is $\SSP(F_2\times F_2)$ $\NP$-complete (cf. Proposition~\ref{pr:ssp_cross})?
\end{itemize}

\section{Acyclic graph word problem}\label{sec:agp}
Let $G$ be a group and $X=\{x_1,\ldots,x_n\}$ a generating set for $G$.

%\medskip
%\noindent
%{\bf The acyclic graph membership problem:} Given $g\in G$ and an acyclic oriented graph $\Gamma$ labeled by words in $X\cup X^{-1}$, decide whether there is a directed path in $\Gamma$ labeled by a word $w$ such that $w=g$ in~$G$.
%
%\medskip
%Note that by exhaustively enumerating all pairs of vertices in $\Gamma$, this problem $\P$-time Cook reduces to the following.

\medskip
\noindent
{\bf The acyclic graph word problem $\AGP(G,X)$\index{$\AGP(G,X)$}:}
Given an acyclic directed graph $\Gamma$ labeled by letters in $X\cup X^{-1}\cup \{\varepsilon\}$ with two marked vertices, $\alpha$ and $\omega$, decide whether there is an oriented path in $\Gamma$ from $\alpha$ to $\omega$ labeled by a word $w$ such that $w=1$ in~$G$.

\medskip
An immediate observation is that this problem, like all the problems introduced in Section~\ref{sec:prelim}, is a special case of the uniform membership problem in a rational subset of $G$. We elaborate on that in Section~\ref{sec:free_prod}.

Let graph $\Gamma$ have $n$ vertices and $m$ edges. Define $\size(\Gamma)$ to be $m+n$. Let total word length of labels of edges of $\Gamma$ be $l$. For a given instance of $\AGP(G,X)$, its {\em size} is the value $m+n+l$. With a slight abuse of terminology, we will also sometimes use labels that are words rather than letters in $X\cup X^{-1}\cup\{\varepsilon\}$. Note that given such a graph $\Gamma$ of size $m+n+l$, subdividing its edges we can obtain a graph $\Gamma'$ where edges are labeled by letters, with $m+n+l\le \size(\Gamma')\le (m+l)+(n+l)+l$. Therefore, $1/3\size(\Gamma')\le\size(\Gamma)\le \size(\Gamma')$, so such an abuse of terminology results in distorting the size of an instance of $\AGP(G,X)$ by a factor of at most $3$.

Note that by a standard argument, if $X_1$ and $X_2$ are two finite generating sets for a group $G$, the problems $\AGP(G,X_1)$ and $\AGP(G,X_2)$ are $\P$-time (in fact, linear time) equivalent. In this sense, the complexity  of $\AGP$ in a group $G$ does not depend on the choice of a finite generating set. In the sequel we write $\AGP(G)$ instead of $\AGP(G,X)$, implying an arbitrary finite generating set.
As shown in~\cite{Miasnikov-Nikolaev-Ushakov:2014a}, complexity of algorithmic problems in a group can change dramatically depending on a choice of infinite generating set. We do not consider infinite generating sets in the present work, and refer the reader to the above paper %~\cite{Miasnikov-Nikolaev-Ushakov:2014a}
for details of treating those.

The above definition of $\AGP$ can be given for an arbitrary finitely generated monoid $G$, rather than a group. The input in such case is a pair $(\Gamma,w_0)$ consisting of a directed graph $\Gamma$ labeled by letters in $X\cup\{\varepsilon\}$, and a word $w_0$ in $X$. The problem asks to decide if a word $w$ such that $w=w_0$ in $G$ is readable as a label of a path in $\Gamma$ from $\alpha$ to $\omega$. Given the immediate linear equivalence of the two formulations in the case of a group, we use the same notation $\AGP(G)$ for this problem. Note that the problems introduced in Section~\ref{sec:prelim} (subset sum problem, knapsack problem, etc) also can be formulated for a monoid.

%When infinite generating sets are concerned, the situation is more involved. We refer the reader to~\cite{Miasnikov-Nikolaev-Ushakov:2014a} for details regarding treatment of infinite generating sets.

We make a note of the following obvious property of $\AGP$.

\begin{proposition}\label{pr:subgroup}
Let $G$ be a finitely generated monoid and $H\le G$ its finitely generated submonoid. Then $\AGP(H)$ is $\P$-time reducible to $\AGP(G)$. In particular,
\begin{enumerate}
\item If $\AGP(H)$ is $\NP$-hard then $\AGP(G)$ is $\NP$-hard.
\item If $\AGP(G)\in \P$ then $\AGP(H)\in \P$.
\end{enumerate}
\end{proposition}

Methods used in~\cite{Miasnikov-Nikolaev-Ushakov:2014a} to treat $\SSP$ and $\BSMP$ in hyperbolic and nilpotent groups can be easily adjusted to treat $\AGP$ as well, as we see in the two following statements.
\begin{proposition}\label{pr:agp_nilp}
$\AGP(G)\in\P$ for every finitely generated virtually nilpotent group $G$.
\end{proposition}
\begin{proof} The statement follows from Theorem~\ref{th:kill_z} below.
%Follows immediately from polynomial growth of virtually nilpotent groups~\cite{Wolf}, similar to the proof of Theorem~3.3 in~\cite{Miasnikov-Nikolaev-Ushakov:2014a}.
\end{proof}

\begin{proposition}\label{pr:agp_hyp}
$\AGP(G)\in\P$ for every finitely generated hyperbolic group $G$.
\end{proposition}
\begin{proof}
The statement follows immediately from~\cite[Proposition 5.5]{Miasnikov-Nikolaev-Ushakov:2014a}.
\end{proof}
Here we would also like to note that since the word problem straightforwardly reduces to $\AGP$, an obvious prerequisite for $\AGP(G)$ to belong to $\NP$ is to have a polynomial time word problem. Therefore, in the context of investigating time complexity of $\AGP$ we are primarily interested in such groups.

We also note that adding or eliminating a direct factor $\mathbb Z$ or, more generally, a finitely generated virtually nilpotent group, does not change complexity of $\AGP$.
%\begin{proposition}\label{pr:kill_z}
%Let $G$ be a finitely generated group. Then $\AGP(G)$ and $\AGP(G\times \mathbb Z)$ are $\P$-time Cook equivalent.
%\end{proposition}
%\begin{proof}
%
%
%Suppose $\Gamma$ is a given acyclic directed graph with edge labels $(w_1,n_1)$, $(w_2,n_2)$, $\ldots$, $(w_k,n_k)$. We compute $N=|n_1|+\ldots+|n_k|$ and consider graph $\Delta$ with vertex set $V(\Gamma)\times [-N,N]$.
%
%{\an Finish the proof!}
%
%\end{proof}

\begin{theorem}\label{th:kill_z}
Let $G$ be a finitely generated monoid and $N$ a finitely generated virtually nilpotent group. Then $\AGP(G)$ and $\AGP(G\times N)$ are $\P$-time equivalent.
\end{theorem}

\begin{proof}
We only have to show the reduction of $\AGP(G\times N)$ to $\AGP(G)$, since the other direction is immediate by Proposition~\ref{pr:subgroup}.

Since the complexity of $\AGP$ in a given monoid does not depend on a finite generating set, we assume that the group $G\times N$ is generated by finitely many elements $(g_1,1_N)$, $(g_2,1_N)$, \ldots, and $(1_G,h_1)$, $(1_G,h_2)$, \ldots. Consider an arbitrary instance $(\Gamma,\alpha,\omega, (\hat g,\hat h))$ of $\AGP(G\times N)$, where $\Gamma = (V,E)$.
%and $w = (g,n)$ is an element in $G\times\mathbb Z$.
Consider a graph $\Gamma^\ast = (V^\ast,E^\ast)$, where $V^\ast = V\times N$ and
\[
E^\ast = \Set{(v,h) \stackrel{g}{\to} (v',hh')}{\mbox{for every }(v,h) \in V^\ast \mbox{ and } v \stackrel{(g,h')}{\to} v' \in E}.
\]
where $(g,h')$ denotes an element of $G\times N$, with $g\in G$, $h'\in \mathbb \{1_N,h_i,h_i^{-1}\}$.
Let $\Gamma'= (V', E')$ be the connected component of $\Gamma^\ast$ containing $(\alpha,1_N) \in V^\ast$
(or the subgraph of $\Gamma^\ast$ induced by all vertices in $\Gamma^\ast$ that can be reached from $(\alpha,1_N)$).
It is easy to see that
\[|V'| \le |V| \cdot B_{|E|},\]
where $B_{|E|}$ is the ball of radius $|E|$ in the Cayley graph of $N$ relative to generators $\{h_i\}$. Since the group $N$ has polynomial growth~\cite{Wolf} and polynomial time decidable word problem (in fact, real time by~\cite{Holt-Rees:2001}), the graph $\Gamma'$ can be constructed in a straightforward way in polynomial time.
Finally, it follows from the construction of $\Gamma'$ that $(\Gamma,\alpha,\omega, (\hat g,\hat h))$ is a positive instance
of $\AGP(G\times N)$ if and only if
$(\Gamma',(\alpha,1_N),(\omega,\hat h), \hat g)$ is a positive instance of $\AGP(G)$.
\end{proof}

The above statement and Proposition~\ref{pr:agp_hyp} provide the following corollary.
\begin{corollary}\label{co:F2xZ}
$\AGP(F_2\times N) \in \P$ for every finitely generated virtually nilpotent group $N$.
\end{corollary}
Here, and everywhere below, $F_2$ denotes the free group of rank $2$.

\subsection{Connection between acyclic graph word problem and knapsack-type problems}
The subset sum problem ($\SSP$), the bounded knapsack problem ($\BKP$), and the bounded submonoid membership problem ($\BSMP$) in a monoid $G$ reduce easily to the acyclic graph word problem ($\AGP$).
\begin{proposition}\label{pr:reduction_to_agp} Let $G$ be a finitely generated monoid.
$\SSP(G),\ \BKP(G),\ \BSMP(G)$ are $\P$-time reducible to $\AGP(G)$.
\end{proposition}
\begin{proof} The proof below is for the case of a group $G$. The case of a monoid $G$ is treated in the same way with obvious adjustments.

Let $w_1,w_2,\ldots, w_k, w$ be an input of $\SSP(G)$. Consider the graph $\Gamma=\Gamma(w_1,\ldots,w_k,w)$ shown in the Figure~\ref{fi:SSP}. We see immediately that $(w_1,\ldots,w_k,w)$ is a positive instance of $\SSP(G)$ if and only if $\Gamma$ is a positive instance of $\AGP(G)$.
\begin{figure}[h]
 \centering
 \includegraphics[width=4.5in]{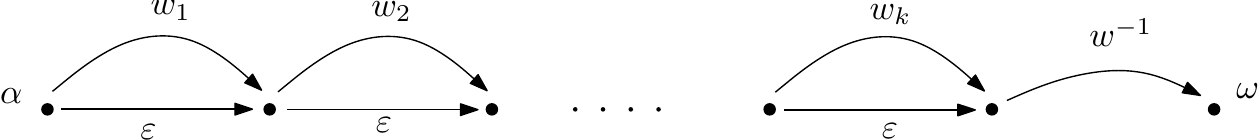}
 \caption{Graph $\Gamma(w_1,w_2,\ldots,w_k,w)$, Proposition~\ref{pr:reduction_to_agp}.}\label{fi:SSP}
\end{figure}
Since $\BKP(G)$ $\P$-time reduces to $\SSP(G)$ (see~\cite{Miasnikov-Nikolaev-Ushakov:2014a}), it is only left to prove that $\BSMP(G)$ reduces to $\AGP(G)$. Indeed, let $(w_1,w_2,\ldots,w_k,w,1^n)$ be an input of $\BSMP(G)$. Consider the graph $\Delta=\Delta(w_1,w_2,\ldots,w_k,w,1^n)$ shown in Figure~\ref{fi:BSMP}.
\begin{figure}[h]
 \centering
 \includegraphics[width=4.5in]{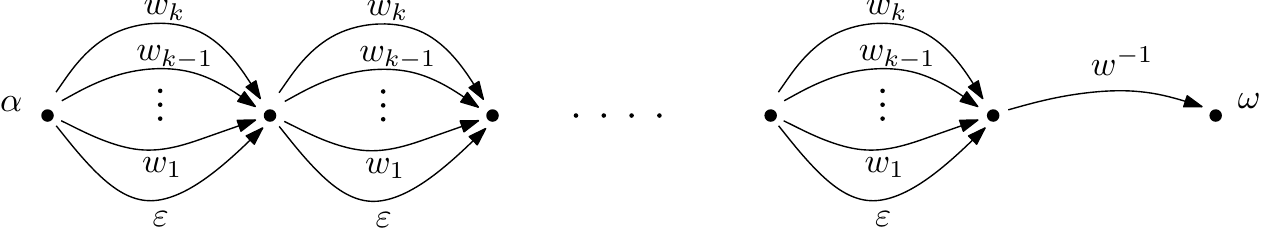}
 \caption{Graph $\Delta(w_1,w_2,\ldots,w_k,w,1^n)$, Proposition~\ref{pr:reduction_to_agp}. There are $n+2$ vertices in the graph.}\label{fi:BSMP}
\end{figure}
It is easy to see that $(w_1,w_2,\ldots,w_k,w,1^n)$ is a positive instance of $\BSMP(G)$ if and only if $\Delta$ is a positive instance of $\AGP$.
\end{proof}

Given Proposition~\ref{pr:reduction_to_agp} and results of~\cite{Miasnikov-Nikolaev-Ushakov:2014a}, we immediately see that $\AGP$ is $\NP$-complete in the following cases;
\begin{itemize}
\item[--] certain metabelian groups (finitely generated free metabelian groups, wreath products of two infinite abelian groups, Baumslag's group $B=\langle a,s,t\mid [a,a^t]=1, [s,t]=1, a^s=aa^t\rangle$, Baumslag--Solitar groups $BS(m,n)$ with $|m|\neq |n|$, $m,n\neq 0$),
\item[--] Thompson's group $F$,
\item[--] $F_2\times F_2$,
\item[--] linear groups $GL(n, \mathbb Z)$ with $n\ge 4$,
\item[--] braid groups $B_n$ with $n\ge 5$ (by~\cite{Makanina}),
\item[--] graph groups whose graph contains an induced square $C_4$.
\end{itemize}
%{\li{How to make this list more readable?} {\an For now we can use any list environment, any publisher will redo it to their standards anyway.}}{\li{For now, I'm happy with this one}}

While it still remains to be seen whether $\AGP$ reduces to either of the problems in Proposition~\ref{pr:reduction_to_agp}, we make note of the following two observations. First, in every case when it is known that those problems are $\P$-time, so is $\AGP$, as shown in Propositions~\ref{pr:agp_nilp} and~\ref{pr:agp_hyp}.
The second observation is that for a given group $G$, $\AGP(G)$ $\P$-time reduces to either of the problems $\SSP(G\ast F_2)$ or $\SSP(G\times F_2)$.

\begin{proposition}\label{pr:agp_to_ssp_star}\label{pr:agp_to_ssp_cross}
Let $G$ be a finitely generated monoid.
\begin{enumerate}
\item  $\AGP(G)$ is $\P$-time reducible to $\SSP(G\ast F_2)$.
\item  $\AGP(G)$ is $\P$-time reducible to $\SSP(G\times F_2)$.
\end{enumerate}
\end{proposition}
\begin{proof}
Let $\Gamma$ be a given directed acyclic graph on $n$ vertices with edges labeled by group words in a generating set $X$ of the group $G$. We start by organizing a topological sorting on $\Gamma$, that is enumerating vertices of $\Gamma$ by symbols $V_1$ through $V_n$ so that if there is a path in $\Gamma$ from $V_i$ to $V_j$ then $i\le j$. This can be done in a time linear in $\size(\Gamma)$ by~\cite{Kahn}. We assume $\alpha=V_1$ and $\omega=V_n$, otherwise discarding unnecessary vertices. We perform a similar ordering of edges, i.e., we enumerate them by symbols $E_1,\ldots, E_m$ so that if there is a path in $\Gamma$ whose first edge is $E_i$ and the last edge $E_j$, then $i\le j$ (considering the derivative graph of $\Gamma$ we see that this can be done in time quadratic in $\size(\Gamma)$). For each edge $E_i$, $1\le i\le m$, denote its label by $u_i$, its origin by $V_{o(i)}$, and its terminus by $V_{t(i)}$. We similarly assume that $o(1)=1$ and $t(m)=n$.

Next, we produce in polynomial time $n$ freely independent elements $v_1,\ldots, v_n$ of the free group $F_2= \langle x, y \rangle$, of which we think as labels of the corresponding vertices $V_1,\ldots, V_n$. For example, $v_j=x^jyx^j$, $j=1,\ldots,n$, suffice.
We claim that
$$g_1 = v_{o(1)}u_1v_{t(1)}^{-1},\ g_2=v_{o(2)}u_2v_{t(2)}^{-1},\ \ldots,\ g_m=v_{o(m)}u_mv_{t(m)}^{-1};\ g=v_{1}v_{n}^{-1}
$$
is a positive instance of $\SSP(G\ast F_2)$ (or $\SSP(G\times F_2)$) if and only if $\Gamma$ is positive instance of $\AGP(G)$. (See Figure~\ref{fi:agp_to_ssp} for an example.)
\begin{figure}[h]
 \centering
 \includegraphics[width=4.5in]{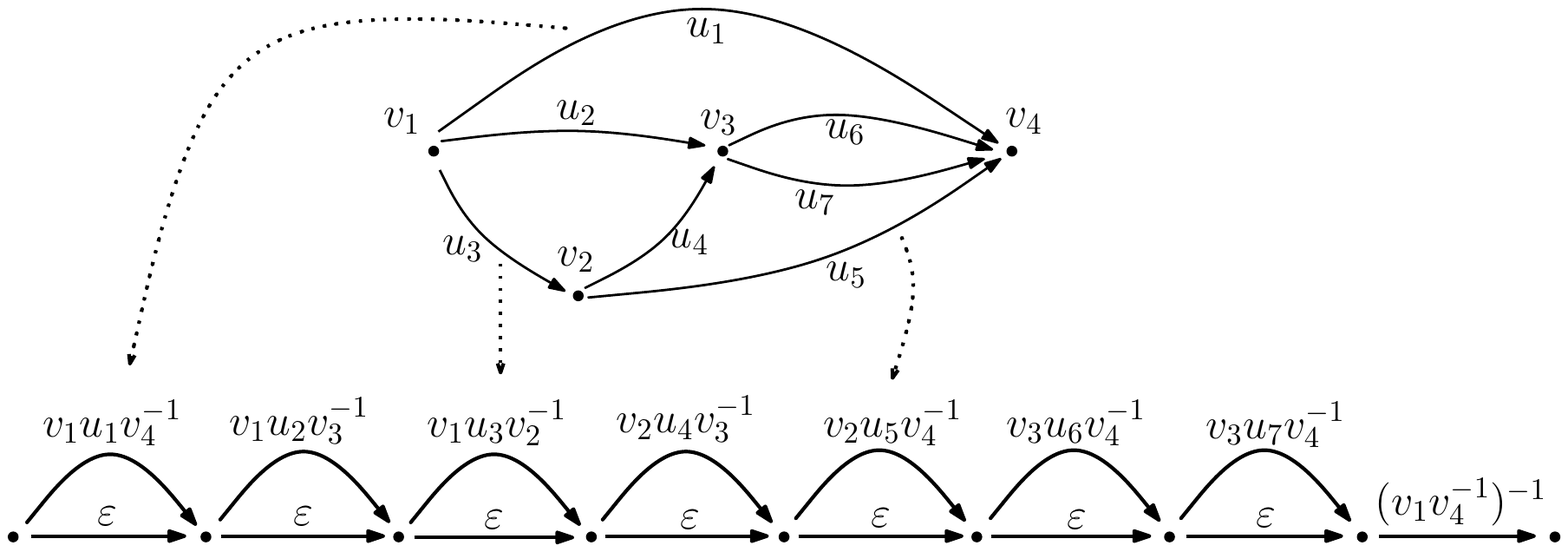}
 \caption{Reduction of $\AGP(G)$ to $\SSP(G\ast F_2)$ and $\SSP(G\times F_2)$. Dotted arrows illustrate the correspondence between edges labeled by $u_i$ and elements $g_i=v_{o(i)}u_iv_{t(i)}^{-1}$.}\label{fi:agp_to_ssp}
\end{figure}
Indeed, suppose there is an edge path $E_{i_1}, \ldots, E_{i_k}$ from $\alpha=V_1$ to $\omega=V_n$ in $\Gamma$ with $u_{i_1}\cdots u_{i_k}=1$ in $G$. Note that since the above sequence of edges is a path, for each $1\le \mu\le k-1$, we have $V_{t(\mu)}=V_{o(\mu+1)}$, so $v_{t(\mu)}=v_{o(\mu+1)}$. Then the same choice of elements $g_{i_1},\ldots, g_{i_k}$ gives
\begin{eqnarray*}
g_{i_1}\cdots g_{i_k}&=&(v_{o(i_1)}u_{i_1}v_{t(i_1)}^{-1})
(v_{o(i_2)}u_{i_2}v_{t(i_2)}^{-1})\cdots (v_{o(i_k)}u_{i_k}v_{t(i_k)}^{-1})\\
&=&v_{o(i_1)}u_{i_1}(v_{t(i_1)}^{-1}
v_{o(i_2)})u_{i_2}(v_{t(i_2)}^{-1}v_{o(i_3)})\cdots v_{o(i_k)}u_{i_k}v_{t(i_k)}^{-1}\\
&=&v_{o(i_1)}u_{i_1}u_{i_2}\cdots u_{i_m}v_{t(i_k)}^{-1}\\
&=&v_{o(i_1)}v_{t(i_k)}^{-1}=v_{1}v_{n}^{-1}=g.
\end{eqnarray*}
In the opposite direction, suppose in $G\ast F_2$ (or $G\times F_2$) the equality
\begin{equation}\label{eq:ssp}
g_{i_1}\cdots g_{i_k}=g,\quad i_1<\ldots<i_k,
\end{equation}
takes place. Consider the $F_2$-component of this equality:
$$
v_{o(i_1)}v_{t(i_1)}^{-1}\cdot v_{o(i_2)}v_{t(i_2)}^{-1}\cdots v_{o(i_k)}v_{t(i_k)}^{-1}=v_1v_n^{-1}.
$$
Since $v_1,\ldots, v_n$ are freely independent, it is easy to see by induction on $k$ that the latter equality only possible if
$$
v_{o(i_1)}=v_1,\ v_{t(i_1)}=v_{o(i_2)},\ v_{t(i_2)}=v_{o(i_3)},\ \ldots,\ v_{t(i_{k-1})}=v_{o(i_k)},\ v_{t(i_k)}=v_n,
$$
i.e. edges $E_{i_1}, E_{i_2},\ldots, E_{i_k}$ form a path from $V_1$ to $V_n$ in $\Gamma$. Further, inspecting the $G$-component of the equality~\eqref{eq:ssp}, we get that $u_{i_1}u_{i_2}\cdots u_{i_k}=1$, as required in $\AGP(G)$.
\end{proof}

%\begin{proposition} %\label{pr:agp_to_ssp_cross}
%Let $G$ be a finitely generated group. Then $\AGP(G)$ is $\P$-time reducible to $\SSP(G\times F_2)$.
%\end{proposition}
%\begin{proof}
%The proof repeats that of Proposition~\ref{pr:agp_to_ssp_star} with $G\times F_2$ instead of $G\ast F_2$.
%\end{proof}

\section{$\AGP$ and $\SSP$ in direct products.}\label{sec:direct_prod}
We show in this section that direct product of groups may change the complexity of the subset sum problem ($\SSP$) and the acyclic graph word problem ($\AGP$) dramatically, in contrast with results of Section~\ref{sec:free_prod}, where we show that free products preserve the complexity of $\AGP$.
\begin{proposition}\label{pr:agp_cross}
There exist groups $G,H$ such that $\AGP(G),\AGP(H)\in\P$, but $\AGP(G\times H)$ is $\NP$-complete.
\end{proposition}
\begin{proof}
It was shown in~\cite[Theorem 7.4]{Miasnikov-Nikolaev-Ushakov:2014a} that $\BSMP(F_2\times F_2)$ is $\NP$-complete. By Proposition~\ref{pr:reduction_to_agp} it follows that $\AGP(F_2\times F_2)$ is $\NP$-complete, while by Proposition~\ref{pr:agp_hyp} $\AGP(F_2)\in\P$.
\end{proof}

A similar statement can be made about $\SSP$ with the help of Proposition~\ref{pr:agp_to_ssp_cross}(2). %as will be clear from the proof of Proposition~\ref{pr:ssp_cross} below.
%Proposition~\ref{pr:agp_to_ssp_cross} is enough to make a similar statement about $\SSP$.
However, in the next proposition we organize reduction of $\BSMP(G)$ to $\SSP(G\times \mathbb Z)$, thus simplifying the ``augmenting'' group, which allows to make a slightly stronger statement about complexity of $\SSP$ in direct products. We remind that the definition of $\P$-time Cook reduction used in the statement below can be found in Section~\ref{sec:prelim}. %Recall that a problem $\Pi_1$ $\P$-time Cook reduces to a problem $\Pi_2$ if there is an algorithm that solves $\Pi_1$ using a polynomial number of calls to a solution of $\Pi_2$, and polynomial time outside of those calls (see Section~\ref{sec:prelim} for the definition).

\begin{proposition}\label{pr:bsmp_to_ssp}
Let $G$ be a finitely generated monoid. Then $\BSMP(G)$ $\P$-time Cook reduces to $\SSP(G\times \mathbb Z)$.
\end{proposition}
\begin{proof} The proof below is for the case of a group $G$. The case of a monoid $G$ is treated in the same way with obvious adjustments.

Let $w_1,w_2,\ldots, w_k, w, 1^n$ be the input of $\BSMP(G)$. We construct graphs $\Gamma_m$, $m=1,\ldots, n$, with edges labeled by elements of $G\times \mathbb Z$ as shown in the Figure~\ref{fi:bsmp_to_ssp}. Note that a path from $\alpha$ to $\omega$ is labeled by a word trivial in $G\times\mathbb Z$ if and only if it passes through exactly $m$ edges labeled by $(w_{i_1},1),\ldots,(w_{i_m},1)$ and $w_{i_1}\cdots w_{i_m}=w$ in $G$. Therefore, the tuple $w_1,\ldots,w_k,1^n$ is a positive instance of $\BSMP(G)$ if and only if at least one of graphs $\Gamma_1,\ldots, \Gamma_n$ is a positive instance of $\SSP(G\times\mathbb Z)$.
\begin{figure}[h]
 \centering
 \includegraphics[width=6in]{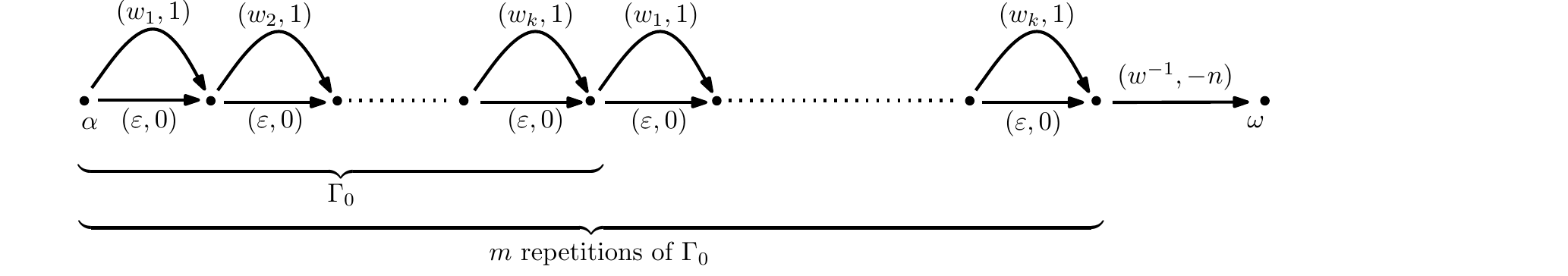}
 \caption{Reduction of $\BSMP(G)$ to $\SSP(G\times \mathbb Z)$. Graph $\Gamma_m$.} \label{fi:bsmp_to_ssp}
\end{figure}
\end{proof}
%{\li{This statement implies a natural question: what about Karp reduction?}} {\an Not convinced this is interesting, unless I am missing something.}{\li{No, you are not. We don't need to mention this anymore}}

We put $G=F_2\times F_2$ in the above Proposition~\ref{pr:bsmp_to_ssp} to obtain the following result.
\begin{proposition}\label{pr:ssp_cross}
$\SSP(F_2\times F_2\times \mathbb Z)$ is $\NP$-complete.
\end{proposition}
\begin{proof}
As we mentioned in the proof of Proposition~\ref{pr:agp_cross}, $\AGP(F_2\times F_2)$ is $\NP$-complete. By Proposition~\ref{pr:bsmp_to_ssp}, the latter $\P$-time Cook reduces to $\SSP((F_2\times F_2)\times\mathbb Z)$). Therefore, the $\SSP(F_2\times F_2\times \mathbb Z)$ is $\NP$-complete.
\end{proof}

The latter proposition answers the question whether direct product preserves polynomial time $\SSP$.

\begin{corollary}\label{cor:NP_complete_ssp_cross}
There exist finitely generated groups $G,H$ such that $\SSP(G)\in\P$, $\SSP(H)\in\P$ but $\SSP(G\times H)$ is $\NP$-complete.
\end{corollary}
\begin{proof}
By Corollary~\ref{co:F2xZ}, $\AGP(F_2\times\mathbb Z)$ is in $\P$. Therefore, $\SSP(F_2)$ and $\SSP(F_2\times\mathbb Z)$ are in $\P$, while
$\SSP(F_2\times (F_2\times \mathbb Z))$ is $\NP$-complete by the above result.
\end{proof}

\begin{corollary}
$\SSP$ is $\NP$-complete in braid groups $B_n$, $n\ge 7$, special linear groups $SL(n,\mathbb Z)$, $n\ge 5$, graph groups whose graph contains the square pyramid~$\boxtimes$ (also called the wheel graph~$W_5$) as an induced subgraph.
\end{corollary}
\begin{proof}
Note that $F_2\times F_2$ embeds in a braid group $B_n$ with $n\ge 5$ by a result of Makanina~\cite{Makanina}. The statement now follows from Propositions~\ref{pr:bsmp_to_ssp} and~\ref{pr:subgroup} since all of the listed groups contain $F_2\times F_2\times \mathbb Z$ as a subgroup.
\end{proof}

It remains to be seen whether $\SSP(F_2\times F_2)$ is $\NP$-complete.

\section{$\AGP$ and rational subset membership problem in free products with finite amalgamated subgroups}\label{sec:free_prod}
Our primary concern in this section is the complexity of acyclic graph word problem ($\AGP$) in free products of groups. However, our approach easily generalizes to free products with amalgamation over a finite subgroup, and to a wider class of problems. We remind that the definition of $\P$-time Cook reduction can be found in Section~\ref{sec:prelim}.
\begin{theorem}\label{th:agp_to_factors}
Let $G,H$ be finitely generated groups, and $C$ be a finite group that embeds in $G$, $H$. Then $\AGP(G*_CH)$ is $\P$-time Cook reducible to $\AGP(G),\AGP(H)$.
\end{theorem}
\begin{proof}
Let $G$ be given by a generating set $X$, and $H$ by $Y$. Let $\Gamma=\Gamma_0$ be the given acyclic graph labeled by $\Sigma=X\cup X^{-1}\cup Y\cup Y^{-1}\cup C$ (we assume that alphabets $X^{\pm 1}$, $Y^{\pm 1}$ are disjoint from $C$). Given the graph $\Gamma_k$, $k\in\mathbb Z$, construct graph $\Gamma_{k+1}$ by adding edges to $\Gamma_k$ as follows.

Consider $\Gamma_{k}'$, the maximal subgraph of $\Gamma_k$ labeled by  $X\cup X^{-1}\cup C$ (i.e., the graph obtained by removing all edges labeled by $Y\cup Y^{-1}$).
For each $c\in C$ and each pair of vertices $v_1,v_2\in V(\Gamma_k)$ from the same connected component of $\Gamma_k$, decide whether a word equal to $c$ in $G$ is readable as a label of an oriented path in $\Gamma_k'$, using the solution to $\AGP(G)$.
For $c\in C$, let $E_c$ be the set of pairs $(v_1,v_2)\in V(\Gamma_k)\times V(\Gamma_k)$ such that the answer to the above question is positive, and there is no edge $v_1\overset{c}{\to}v_2$ in $\Gamma_k.$ Construct the graph $\overline{\Gamma}_k$ by adding edges  $v_1\overset{c}{\to}v_2$, $(v_1,v_2)\in E_c$ to the graph $\Gamma_k$, for all $c\in C$.
Now consider the maximal subgraph $\Gamma_k''$ of $\overline{\Gamma}_k$ labeled by $Y\cup Y^{-1}\cup C$ and perform similar operation using the solution to $\AGP(H)$, obtaining the graph $\overline{\overline{\Gamma}}_k=\Gamma_{k+1}$.
Since there are at most $2|C|\cdot |V(\Gamma)|^2$ possible $c$-edges to be drawn, it follows that $\size(\Gamma_{k+1})<2|C|\size(\Gamma)^2$, and that $\Gamma_k=\Gamma_{k+1}=\ldots$ for some $k=n$, where $n\le 2|C| \size(\Gamma)^2$.

We claim that a word $w$ equal to $1$ in $G*_CH$ is readable from $\alpha$ to $\omega$ in $\Gamma$ if and only if there is an edge $\alpha\overset{\varepsilon}{\to}\omega$ in the graph $\Gamma_{n}$.
Indeed, suppose there is a path in $\Gamma$ and, therefore, in $\Gamma_n$, from $\alpha$ to $\omega$ labeled by a word $w=w_1w_2\cdots w_m$, with $w_j\in \Sigma$ and at least one non-$C$ letter among $w_1,\ldots, w_m$, such that $w=1$ in $G*_CH$.
The normal form theorem for free products with amalgamated subgroup guarantees that $w$ has a subword $w'=w_iw_{i+1}\cdots w_j$ of letters in $X\cup X^{-1}\cup C$ or $Y\cup Y^{-1}\cup C$ with $w'=c\in C$ in $G$ or $H$, respectively, with at least one non-$C$ letter among $w_i,\ldots,w_j$.
Since $\Gamma_n=\Gamma_{n+1}$, the word $w_1\cdots w_{i-1}c w_{j+1}\cdots w_m$ is readable as a label of a path in $\Gamma_{n}$ from $\alpha$ to $\omega$. 
By induction, a word $c_1\cdots c_\ell$, $\ell\le m$, $c_1,\ldots,c_\ell\in C$, is readable as a label of an oriented path in $\Gamma_{n}$ from $\alpha$ to $\omega$.
By the construction, $\Gamma_{n+1}=\Gamma_{n}$ contains an edge $\alpha\overset{\varepsilon}{\to}\omega$.
The converse direction of the claim is evident.
%Draw $\varepsilon$-edges in the graph.
\end{proof}

\begin{corollary}\label{cor:ptime_agp}
Let $G,H$ be finitely generated groups, and $C$ be a finite group that embeds in $G$, $H$. If $\AGP(G)$, $\AGP(H)\in\P$ then $\AGP(G*_CH)\in\P$.
\end{corollary}

\begin{corollary}\label{cor:ptme_ssp_free}
Subset sum ($\SSP$), bounded knapsack ($\BKP$), bounded submonoid membership ($\BSMP$), and acyclic graph word ($\AGP$) problems are polynomial time decidable in any finite free product with finite amalgamations of finitely generated virtually nilpotent and hyperbolic groups.
\end{corollary}

Theorem~\ref{th:agp_to_factors} can be generalized to apply in the following setting. We say that a family $\mathcal F$ of finite directed graphs is {\em progressive} if it is closed under the following operations:
\begin{enumerate}
\item taking subgraphs,
\item adding shortcuts (i.e., drawing an edge from the origin of an oriented path to its terminus), and 
\item appending a hanging oriented path by its origin. 
\end{enumerate}
Two proper examples of such families are acyclic graphs, and graphs stratifiable in a sense that vertices can be split into subsets $V_1,\ldots, V_n$ so that edges only lead from $V_k$ to $V_{\ge k}$. We say that a subset of a group $G$ is $\mathcal F$-rational if it can be given by a finite (non-deterministic) automaton in $\mathcal F$. Finally, by {\em $\mathcal F$-uniform rational subset membership problem} for $G$ we mean the problem of establishing, given an $\mathcal F$-rational subset of a group $G$ and a word as an input, whether the subset contains an element equal to the given word in $G$. In this terminology, $\AGP(G)$ is precisely the $\mathcal F_{\mathrm{acyc}}$-uniform rational subset membership problem for $G$, where $\mathcal F_{\mathrm{acyc}}$ is the family of acyclic graphs. % Also, observe that uniform rational subset membership problem for a group $G$ is linearly equivalent to $\mathcal F$-uniform trivial element problem, where $\mathcal F$ is the family of all finite directed graphs.

\begin{theorem}\label{th:generalized_agp_free}
Let $\mathcal F$ be a progressive family of directed graphs. Let $G,H$ be finitely generated groups, and $C$ be a finite group that embeds in $G$, $H$. Then $\mathcal F$-uniform rational subset membership problem for $G*_CH$ is $\P$-time Cook reducible to that for $G$ and $H$.
\end{theorem}
\begin{proof}
Given an automaton $\Gamma$ and a word $w$ as an input, we start by forming an automaton $\Gamma_w$ by appending a hanging path labeled by $w^{-1}$ at every accepting state of $\Gamma$. The procedure then repeats that in the proof of Theorem~\ref{th:agp_to_factors}, with obvious minor adjustments. 
\end{proof}
In particular, uniform rational subset membership problem for $G*_CH$, with $C$ finite, is $\P$-time Cook reducible to that in the factors $G$ and $H$, cf. decidability results in the case of graphs of groups with finite edge groups in~\cite{Kambites-Silva-Steinberg}.

Given that the uniform rational subset membership problem is $\P$-time decidable for free abelian groups~\cite{Lohrey-survey}, and that $\P$-time decidability of the same problem carries from a finite index subgroup~\cite{Grunschlag}, the above theorem gives the following corollary.

\begin{corollary} Let $G$ be a finite free product with finite amalgamations of finitely generated abelian groups. The rational subset membership problem in $G$ is decidable in polynomial time.
\end{corollary}

\section{Knapsack problem in free products}\label{sec:knapsack}
As examples cited in Section~\ref{sec:prelim} show, the (unbounded) knapsack problem ($\KP$) in general does not reduce to its bounded version ($\BKP$), nor to acyclic graph word problem ($\AGP$). However, it was shown in~\cite{Miasnikov-Nikolaev-Ushakov:2014a} that in the case of hyperbolic groups there is indeed such reduction. In light of results of Section~\ref{sec:free_prod}, it is natural to ask whether a similar reduction takes place for free products of groups.

For an element $f$ of a free product of groups $G*H$, let $\|f\|$ denote the {\em syllable length} of $f$, i.e. its geodesic length in generators $G\cup H$ of $G*H$. We say that $f$ is {\em \dumb} if $f$ can be conjugated by an element of $G*H$ into $G\cup H$, i.e.
\begin{equation}\label{eq:simple}
f=u^{-1}f'u,
\end{equation}
where $u\in G*H, \|f'\|\le 1$. Otherwise, we say that $f$ is {\em \nondumb}.

\begin{lemma}\label{le:power}
Let $G,H$ be groups and let an element $f\in G\ast H$ be \nondumb.
%, $f\notin\bigcup_{u\in G*H} G^u\cup H^u$.
Then there are $a_1,\ldots,a_r\in G\cup H$, $b_1,\ldots,b_s\in G\cup H$, $c_1,\ldots,c_t\in G\cup H$, $r+t\le 2\|f\|$, $s\le \|f\|$, such that for every integer $n\ge 3$ the normal form of $f^n$ is
$$
a_1\cdots a_r(b_1\cdots b_s)^{n-2}c_1\cdots c_t.
$$
\end{lemma}
\begin{proof} Follows from the normal form theorem for free products.
\end{proof}

The statement below is, in some sense, a strengthened big-powers condition for a free product.
\begin{proposition}\label{pr:big_power}
Let $p(x)$ be the polynomial $p(x)=x^2+4x$. Let $G,H$ be groups. If for $f_1,f_2,\ldots, f_m, f\in G*H$ there exist non-negative integers $n_1,n_2,\ldots,n_m\in\mathbb Z$ such that
\begin{equation}\label{eq:big_powers}
f_1^{n_1}f_2^{n_2}\cdots f_m^{n_m}=f
\end{equation}
then there exist such integers $n_1,\ldots,n_m$ with
$$n_i\le p(\|f_1\|+\|f_2\|+\ldots+\|f_m\|+\|f\|)\mbox{ whenever } f_i\mbox{ is \nondumb}.
$$
%for each $i$ such that $f_i$ is \nondumb.
\end{proposition}
\begin{proof} For convenience, denote $\|f_1\|+\|f_2\|+\ldots+\|f_m\|+\|f\|=N$ and $f=f_0$, $-1=n_0$.
For each $f_i$, $i=0,\ldots,m$, we represent $f_i^{n_i}$ by its normal form $w_i$, using Lemma~\ref{le:power} whenever $f_i$ is {\nondumb} and $n_i\ge 3$. Suppose that for some {\nondumb} $f_i$, $n_i > N^2+4N$. We inspect the path traversed by the word $w_1w_2\ldots w_mw_0$  in the Cayley graph of $G*H$ with respect to generators $G\cup H$. Since this word corresponds to the trivial group element, the path must be a loop and thus the word $w_i$ splits in at most $m$ pieces, each piece mutually canceling with a subword of $w_j$, $j\neq i$. Let $w_i=ab^{n_i-2}c$ as in Lemma~\ref{le:power}. Then by pigeonhole principle at least one piece contains at least $(n_i-2-m)/m\ge N+1$ copies of the word $b$. Observe that $f_j$ is {\nondumb} and $n_j\ge 3$, otherwise $\|f_j^{n_j}\|\le 2N$, so it cannot mutually cancel with $b^{N+1}$ since $\|b^{N+1}\|\ge 2N+2$. Let $w_j=a'{b'}^{n_j-2}c'$. We note that $\|b\|\le \|f_i\|\le N$, $\|b'\|\le \|f_j\|\le N$, so $\|b'\|\le N$ copies of $b$ mutually cancel with $\|b'\|\cdot\|b\|/\|b'\|=\|b\|$ copies of $b'$ (up to a cyclic shift). Therefore replacing $n_i,n_j$ with $n_i-\|b'\|$ and $n_j-\|b\|$, respectively, preserves equality~\eqref{eq:big_powers}. Iterating this process, we may find numbers $n_1,\ldots, n_m$ that deliver equality~\eqref{eq:big_powers} such that $n_i\le p(N)=N^2+4N$ whenever $f_i$ is \nondumb.
\end{proof}
We observe that $p(N)=N^2+4N$ in the above proposition is monotone on $\mathbb{Z}_{\ge 0}$.

We say that $G$ is a {\em\pbk} group if there is a polynomial $q$ such that any instance $(g_1,\ldots,g_k,g)$ of $\KP(G)$ is positive if and only if the instance $(g_1,\ldots,g_k,g, q(N))$ of $\BKP(G)$ is positive, where $N$ is the total length of $g_1,\ldots, g_k,g$. It is easy to see that this notion is independent of a choice of a finite generating set for~$G$.

\begin{proposition}\label{pr:kp_to_bkp}
If $G,H$ are {\pbk} groups then $G*H$ is a {\pbk} group.
\end{proposition}
\begin{proof}
Let $f_1,f_2,\ldots,f_m,f\in G*H$ be an input of $\KP(G*H)$. For each $f_i$, normal form of $f_i^{n_i}$, $n_i\ge 3$, is given by Lemma~\ref{le:power} or by simplicity of $f_i$. Suppose some $n_1,n_2,\ldots,n_m$ provide a solution to $\KP$, i.e. $f_1^{n_1}\cdots f_m^{n_m}f^{-1}=1$ in $G*H$. Representing $f^{-1}$ and each $f_i^{n_i}$ by their normal forms and combining like terms we obtain (without loss of generality) a product
\begin{equation}\label{eq:norm_form}
g_1h_1g_2h_2\cdots g_\ell h_\ell=1,
\end{equation}
where each $g_i\in G$ and each $h_i\in H$, and we may assume $\ell\le \|f\|+m\cdot (\max\{\|f_i\|\})\cdot p(\max\{\|f_i\|\})$ by Proposition~\ref{pr:big_power}, so $\ell\le N^2p(N)$ where $N$ is the total length of the input (i.e. the sum of word lengths of $f_1,$ \ldots, $f_m,$ $f$). Further, since the product in~\eqref{eq:norm_form} represents the trivial element, it can be reduced to $1$ by a series of eliminations of trivial syllables and combining like terms:
\begin{align*}
g_1h_1g_2h_2\cdots g_\ell h_\ell&=g_1^{(1)}h_1^{(1)}\cdots g_\ell^{(1)}h_\ell^{(1)}\\
%&=g_1^{(1)}h_1^{(1)}\cdots g_\ell^{(1)}h_\ell^{(1)}\\
&=g_1^{(2)}h_1^{(2)}\cdots g_{\ell-1}^{(2)}h_{\ell-1}^{(2)}\\
&=\ldots=g_1^{(\ell)}h_1^{(\ell)}=1,\\
\end{align*}
where the product labeled by $(j)$ is obtained from the one labeled by $(j-1)$ by a single elimination of a trivial term and combining the two (cyclically) neighboring terms. Observe that each $g_i^{(j)}$ is (up to a cyclic shift) a product of the form $g_\alpha g_{\alpha+1}\cdots g_\beta$; similarly for $h_i^{(j)}$. Therefore, each $g_i^{(j)}$ is of the form
\begin{equation}\label{eq:g_ij}
g_i^{(j)}=d_{1j}^{\delta_{1j}} d_{2j}^{\delta_{2j}}\cdots d_{k_{ij},j}^{\delta_{k_{ij},j}},
\end{equation}
where each $d_{\mu j}$, $1\le \mu\le k_{ij}$ is either
\begin{itemize}
\item[(NS)] one of the syllables involved in Lemma~\ref{le:power}, or normal form of some  $f_i$ or $f_i^2$, or ``u'' part of~\eqref{eq:simple}, in which case $\delta_{\mu j}=1$, or
\item[(S)]\label{li:simple} the syllable $f'$ in~\eqref{eq:simple} for some $f_\nu$, in which case $\delta_{\mu j}=n_\nu$.
\end{itemize}
On the one hand, the total amount of syllables $d_{\mu j}$ involved in~\eqref{eq:g_ij} for a fixed $j$ is $k_{1j}+k_{2j}+\cdots +k_{\ell-j+1,j}$. On the other hand, it cannot exceed $k_{11}+k_{21}+\cdots +k_{\ell 1}$
since $g_1^{(j)}$, $g_2^{(j)},$ \ldots, $g_{\ell-j+1}^{(j)}$ are obtained by eliminating and combining elements $g_1$, $g_2$, \ldots, $g_l$. Taking into account that each $k_{i1}\le m+1\le N$, we obtain $k_{1j}+k_{2j}+\cdots +k_{\ell-j+1,j}\le \ell N\le N^3p(N)$.

%total word length of the collection $g_1,\ldots, g_\ell$,

Now, given the equality $g_i^{(j)}=1$ for some $i,j$, if the option~(S) takes place for any $1\le \mu\le k_{ij}$, then the right hand side of the corresponding equality~\eqref{eq:g_ij} can be represented (via a standard conjugation procedure) as a positive instance of the $\KP(G)$ with input of length bounded by $(N k_{ij})^2\le N^8p^2(N)$. Since $G$ is a {\pbk} group, we may assume that every $n_\nu$ that occurs as some $\delta_{\mu j}$ in some $g_i^{(j)}$ is bounded by a polynomial $p_G(N)$. Similar argument takes place for the $H$-syllables $h_i^{(j)}$, resulting in a polynomial bound $p_H(N)$.

It is only left to note that since for every $1\le \nu\le m$, either $f_\nu$ is non-simple and then $n_\nu$ is bounded by $p(N)$, or it is simple and then $n_\nu$ is bounded by $p_G(N)$ or $p_H(N)$, so every $n_\nu$ is bounded by $p(N)+p_G(N)+p_H(N)$.
%we note that each $g_i$ has the form
% % does this really give algorithmic reduction? Do we have to be able to find p_G,p_H effectively?
\end{proof}

\begin{theorem}\label{th:ptime_kp}
If $G,H$ are {\pbk} groups such that $\AGP(G)$, $\AGP(H)\in\P$ then $\KP(G*H)\in\P$.
\end{theorem}
\begin{proof} By Proposition~\ref{pr:kp_to_bkp} $\KP(G*H)$ is $\P$-time reducible to $\BKP(G*H)$. In turn, the latter is $\P$-time reducible to $\AGP(G*H)$ by Proposition~\ref{pr:reduction_to_agp}. Finally, $\AGP(G*H)\in\P$ by Corollary~\ref{cor:ptime_agp}.
\end{proof}

\begin{corollary}
$\KP$ is polynomial time decidable in free products of finitely generated abelian and hyperbolic groups in any finite number.
\end{corollary}
\begin{proof}
By~\cite[Theorem 6.7]{Miasnikov-Nikolaev-Ushakov:2014a}, hyperbolic groups are {\pbk} groups. By~\cite{Borosh-Treybig}, so are finitely generated abelian groups. The statement follows by Theorem~\ref{th:ptime_kp}.
\end{proof}

The authors are grateful to A.~Myasnikov for bringing their attention to the problem and for insightful advice and discussions, and to the anonymous referees for their astute observations and helpful suggestions.
%\bibliography{../../main_bibliography}
\bibliography{main_bibliography}

\begin{thebibliography}{10}

\bibitem{Borosh-Treybig}
I.~{Borosh} and L.~B. {Treybig}.
\newblock Bounds on positive integral solutions of linear {D}iophantine
  equations.
\newblock {\em Proc. Amer. Math. Soc.}, 55(2):299--304, 1976.

\bibitem{Grunschlag}
Z.~{Grunschlag}.
\newblock {\em {Algorithms in Geometric Group Theory}}.
\newblock PhD thesis, University of California at Berkley, 1999.

\bibitem{Holt-Rees:2001}
D.~{Holt} and S.~{Rees}.
\newblock {Solving the word problem in real time}.
\newblock {\em J. London Math. Soc.}, 63(2):623--639, 2001.

\bibitem{Kahn}
A.~B. {Kahn}.
\newblock Topological sorting of large networks.
\newblock {\em Commun. ACM}, 5(11):558--562, 1962.

\bibitem{Kambites-Silva-Steinberg}
Mark {Kambites}, Pedro~V. {Silva}, and Benjamin {Steinberg}.
\newblock {On the rational subset problem for groups}.
\newblock {\em Journal of Algebra}, 309(2):622--639, 2007.

\bibitem{Lohrey-survey}
M.~{Lohrey}.
\newblock {The Rational Subset Membership Problem for Groups: A Survey}.
\newblock 2013.

\bibitem{Lohrey:2015}
Markus {Lohrey}.
\newblock {Rational subsets of unitriangular groups}.
\newblock {\em International Journal of Algebra and Computation},
  25(01n02):113--121, 2015.

\bibitem{Lyndon-Schupp:2001}
R.~{Lyndon} and P.~{Schupp}.
\newblock {\em Combinatorial Group Theory}.
\newblock Classics in Mathematics. Springer, 2001.

\bibitem{Makanina}
T.~A. {Makanina}.
\newblock Occurrence problem for braid groups $\mathfrak{B}_{n + 1}$ with
  $n+1\ge 5$.
\newblock {\em Mathematical notes of the Academy of Sciences of the USSR},
  29(1), 1981.

\bibitem{MSU_book:2011}
A.~G. {Miasnikov}, V.~{Shpilrain}, and A.~{Ushakov}.
\newblock {\em Non-Commutative Cryptography and Complexity of Group-Theoretic
  Problems}.
\newblock Mathematical Surveys and Monographs. AMS, 2011.

\bibitem{Mikhailova}
K.~A. {Mikhailova}.
\newblock {The occurrence problem for direct products of groups}.
\newblock {\em Dokl. Akad. Nauk SSSR}, 119:1103--1105, 1958.

\bibitem{Mikhailova_68}
K.~A. {Mikhailova}.
\newblock {The occurrence problem for free products of groups}.
\newblock {\em Math. USSR-Sb.}, 4:181--190, 1968.

\bibitem{Misch-Treyer}
A.~{Mishchenko} and A.~{Treyer}.
\newblock Knapsack problems in nilpotent groups (preprint).
\newblock 2014.

\bibitem{Miasnikov-Nikolaev-Ushakov:2014a}
Alexei~G. {Myasnikov}, Andrey {Nikolaev}, and Alexander {Ushakov}.
\newblock {Knapsack problems in groups}.
\newblock {\em Math. Comp.}, 84(292):987--1016, 2015.

\bibitem{Olsh-Sapir}
A.~Yu. {Olshanskii} and M.~V. {Sapir}.
\newblock {Length functions on subgroups in finitely presented groups}.
\newblock In Y.~Paek, D.L. Johnson, and A.C. Kim, editors, {\em Groups--Korea
  '98: Proceedings of the International Conference Held at Pusan National
  University, Pusan, Korea, August 10-16, 1998}, De Gruyter Proceedings in
  Mathematics Series, 2000.

\bibitem{Wolf}
J.~{Wolf}.
\newblock {Growth of finitely generated solvable groups and curvature of
  Riemannian manifolds}.
\newblock {\em Journal of Differential Geometry}, 2:421--446, 1968.

\end{thebibliography}
%\printindex

\end{document}